\documentclass[12pt,a4paper]{amsart}

\usepackage{verbatim}
\usepackage{amssymb}
\usepackage{latexsym}
\usepackage{amscd}
\usepackage{enumerate}
\usepackage{amsmath}
\usepackage{amsxtra}

\newtheorem{theo}{Theorem}[section]

\newtheorem{cor}[theo]{Corollary}
\newtheorem{exa}[theo]{Example}
\newtheorem{lem}[theo]{Lemma}
\newtheorem{prop}[theo]{Proposition}
\newtheorem{rem}[theo]{Remark}

\newtheorem{em-bez}[theo]{Notation}

\newtheorem{em-ex}[theo]{Examples and Remarks}

\newtheorem{em-remarks}[theo]{}

\newtheorem{em-re}[theo]{Remark}

\newtheorem{em-defi}[theo]{Definition}
\newenvironment{defi}{\begin{em-defi}\em}{\end{em-defi}}

\newcommand{\converges}[1]{\mathop{\overset{{\scriptscriptstyle\text{#1}}}{\longrightarrow}}}

\newcommand{\sss}[1]{{\scriptscriptstyle #1}}

\newcommand\restr[2]{{
  \left.\kern-\nulldelimiterspace 
  #1 
  \vphantom{\mid} 
  \right|_{#2} 
  }}

\usepackage{color}    

\begin{document}

\title{Lattice uniformities inducing unbounded convergence}

\author{Kevin Abela}
\address{
Kevin Abela \\
Department of Mathematics \\
Faculty of Science \\
University of Malta \\
Msida MSD 2080, Malta} \email {kevin.abela.11@um.edu.mt}

\author{Emmanuel Chetcuti}
\address{
Emmanuel Chetcuti,
Department of Mathematics\\
Faculty of Science\\
University of Malta\\
Msida MSD 2080  Malta} \email {emanuel.chetcuti@um.edu.mt}

\author{Hans Weber}
\address{Hans Weber\\
Dipartimento di Scienze Matematiche, Informatiche e
           Fisiche\\
Universit\`a degli Studi di Udine\\
I-33100 Udine\\
Italia}
\email{hans.weber@uniud.it}

\date{\today}
\begin{abstract}
 A net $(x_\gamma)_{\gamma\in\Gamma}$ in a locally solid Riesz space $(X,\tau)$ is said to be unbounded $\tau$-convergent to $x$ if $|x_\gamma-x|\wedge u\mathop{\overset{\tau}{\longrightarrow}} 0$ for all $u\in X_+$.   We recall that there is a locally solid linear topology $\mathfrak{u}\tau$ on $X$ such that unbounded $\tau$-convergence coincides with $\mathfrak{u}\tau$-convergence, and moreover, $\mathfrak{u}\tau$ is characterised as the weakest locally solid linear topology which coincides with $\tau$ on all order bounded subsets.     It is with this motivation that we introduce, for a  uniform lattice $(L,u)$, the weakest lattice uniformity $u^\ast$ on $L$ that coincides with $u$ on all the order bounded subsets of $L$.  
 It is shown that if $u$ is the uniformity induced by the topology of a locally solid Riesz space $(X,\tau)$, then the $u^*$-topology coincides with $\mathfrak{u}\tau$. This allows comparing results of this paper with earlier results on unbounded $\tau$-convergence.  It will be seen  that despite the fact that in the setup of uniform lattices most of the machinery used in the techniques of \cite{T} is lacking, the concept of `unbounded convergence' well fittingly generalizes to uniform lattices.  We shall also answer Questions 2.13, 3.3, 5.10 of \cite{T} and Question 18.51 of \cite{TaylorThesis}.
\end{abstract}
\maketitle

\section{Introduction}
 Uniform lattices were introduced in \cite{W91,W93} as a generalization of topological Boolean rings (= Boolean rings endowed with an FN-topology) and of locally solid Riesz spaces or, more general, of locally solid $\ell$-groups. A uniform lattice is a lattice with a lattice uniformity, i.e. a uniformity making the lattice operations $\vee$ and $\wedge$ uniformly continuous. The uniformity induced by a group topology $\tau$ on a Boolean ring is a lattice uniformity iff $\tau$ is an FN-topology (in the sense of \cite{D}); the right (left or two-side) uniformity induced by a group topology $\tau$ on an $\ell$-group is a lattice uniformity iff $\tau$ is locally solid (see \cite[Propositions 1.1.7 and 1.1.8]{W91}).

Recently several authors have studied in a locally solid Riesz space $(X,\tau)$ the concept of  `unbounded $\tau$-convergence' of a net $(x_\alpha)_{\alpha\in A}$ to $x$,
i.e. $|x_\alpha-x|\wedge u\mathop{\overset{\tau}{\longrightarrow}} 0$ for all $u\in X_+$. `Unbounded convergence' was studied in \cite{DOT, KMT, KT} when $\tau$ is the norm topology on a Banach lattice, and in \cite{T, DaEmMa2018}, more generally, when $\tau$ is a locally solid linear topology on a Riesz space. Taylor \cite[Theorem 2.3]{T} showed that there is a locally solid linear topology $\mathfrak{u}\tau$ on $X$ such that unbounded $\tau$-convergence coincides with $\mathfrak{u}\tau$-convergence. It can be shown that $\mathfrak{u}\tau$ is the weakest locally solid linear topology which coincides with $\tau$ on all order bounded subsets (cf. Corollary \ref{weakest}, Theorem \ref{t}). 

This is the motivation to introduce here for a uniform lattice $(L,u)$ the weakest lattice uniformity $u^*$ on $L$ which coincides with $u$ on all order bounded subsets of $L$. The $u^*$-topology is then the weakest lattice topology (i.e. making $\vee$ and $\wedge$ continuous) which coincides with the $u$-topology on all order bounded subsets of $L$.
If $u$ is the uniformity induced by the topology of a locally solid Riesz space $(X,\tau)$, then the $u^*$-topology coincides with $\mathfrak{u}\tau$ (see Theorem \ref{cf}). This allows comparing results of this paper with earlier results on unbounded $\tau$-convergence.
Apart from showing how the concept of `unbounded convergence' well fittingly generalizes to uniform lattices, we shall also answer Questions 2.13, 3.3, 5.10 of \cite{T} and Question 18.51 of \cite{TaylorThesis}.


The paper is organized as follows.  In Section 2 we consider two modes of order convergence (Definition \ref{defi}) and prove a result (Theorem \ref{t3}) that allows us to answer \cite[Question 18.51]{TaylorThesis} and \cite[Question 5.10]{T} (see Corollaries \ref{c3} and \ref{c5}).  Section 3 is a summary of tools on uniform lattices
 that are most relevant to our work.
 In Section 4 we extend the notion of unbounded convergence to uniform lattices and show that our definition is  a `faithful' generalization to the one that has been studied in the existing literature on locally solid Riesz spaces.  Section 5 is devoted to the study of the uniformity $u^*$ on sublattices. In particular, Theorems \ref{densegroup} and \ref{t4} answer Questions 3.3 and 2.13 of \cite{T}.

\section{Order convergence, Order topology and Order continuous functions}

Let $(P,\le)$ be a partially ordered set.  A subset $D$ of $P$ is \emph{directed} provided it is nonempty and every finite subset of $D$ has an upper bound in $D$.  Dually, a nonempty subset $F$ of $P$ is \emph{filtered} if every finite subset of $F$ has a lower bound in $F$. If every subset of $P$ that is bounded from above has a supremum and every subset of $P$ that is bounded from below has an infimum, $P$ is said to be \emph{Dedekind complete} (or conditionally complete). $P$ is \emph{complete} if every of its subsets has a supremum and infimum.

Let $(x_\gamma)_{\gamma\in \Gamma}$ be a net in a set $X$. If $\Gamma'$ is some directed set and $\varphi:\Gamma'\to\Gamma$ is increasing and final\footnote{i.e. for every $\gamma\in \Gamma$ there exists $\gamma'\in\Gamma'$ such that $\varphi(\gamma')\ge \gamma$}, then the net $(x_{\varphi(\gamma')})_{\gamma'\in\Gamma'}$ is called a \emph{subnet} of $(x_\gamma)_{\gamma\in \Gamma}$.
Let $P$ be a general property of nets. If there exists a $\gamma_0\in \Gamma$ such the subnet
$(x_\gamma)_{\Gamma\ni\gamma\ge\gamma_0}$ has the property $P$ , then we say that  \emph{$(x_\gamma)_{\gamma\in\Gamma}$ satisfies $P$ eventually}.
If $(x_\gamma)_{\gamma\in\Gamma}$ is  increasing, its supremum exists and equals $x$, we write $x_\gamma\uparrow x$.  Dually, $x_\gamma\downarrow x$ means that the net $(x_\gamma)_{\gamma\in\Gamma}$ is  decreasing with infimum equal to $x$.

 There are distinct definitions for order convergence that were  used for various applications.  We shall here be concerned with the following two modes of order convergence, namely $O_1$- and $O_2$-convergence.
For every $a,b\in P$  define $[a,b]:=\{x\in P:a\le x\le b\}$.

 \begin{defi}\label{defi}
   Let $(x_\gamma)_{\gamma\in\Gamma}$ be a net and  $x$ a point in a poset $P$.
   \begin{enumerate}[{\rm(i)}]
     \item $(x_\gamma)_{\gamma\in\Gamma}$ is said to \emph{$O_1$-converge} to $x$ in $P$ if there exist two nets $(y_\gamma)_{\gamma\in\Gamma}$, $(z_\gamma)_{\gamma\in\Gamma}$ in $P$ such that eventually $y_\gamma\le x_\gamma\le z_\gamma$, $y_\gamma\uparrow x$ and $z_\gamma\downarrow x$.
     \item $(x_\gamma)_{\gamma\in \Gamma}$ is said to \emph{$O_2$-converge} to $x$ in $P$ if  there exists a directed subset $M\subseteq P$, and a filtered subset $N\subseteq P$,  such that $\sup M=\inf N=x$, and for every $(m,n)\in M\times N$ the net is eventually contained in $[m,n]$.
   \end{enumerate}
 \end{defi}

 The notion of $O_1$-convergence can be traced back to Birkhoff \cite{Birkhoff1940} and Kantorovich \cite{Kantorovich1935}.  In contrast to $O_1$-convergence, in $O_2$-convergence, the controlling nets are not indexed by the same directed set of the original net.  $O_2$-convergence is due to McShane \cite{McShane1953}.
 It is easy to see for $i\in\{1,2\}$ that:
  \begin{itemize}
  \item For a net $(x_\gamma)_{\gamma\in\Gamma}$, and point $x$, $x_\gamma\converges{\text{$O_1$}} x\ \Rightarrow\  x_\gamma\converges{\text{$O_2$}}x$.
 \item Every $O_i$-convergent net is eventually order bounded and every eventually constant net is $O_1$-convergent to its eventual value.
 \item Every  subnet of an $O_i$-convergent net is $O_i$-convergent to the same limit.
 \item For a monotone net the two types of order convergence coincide.
\item When $P$ is a Dedekind complete lattice, the expression `$(x_\gamma)_{\gamma\in\Gamma}$ $O_i$-converges to $x$'  conveys the intuitive meaning\footnote{Note that for an $O_i$-convergent net in a Dedekind complete lattice, the  suprema/infima involved in (\ref{e3}) exist eventually.}
\begin{equation}\label{e3}
  \liminf_\gamma x_\gamma=\sup_{\gamma\in\Gamma}\inf_{\gamma'\ge \gamma} x_{\gamma'}=x=\inf_{\gamma\in\Gamma}\sup_{\gamma'\ge \gamma} x_{\gamma'}=\limsup_\gamma x_\gamma.
\end{equation}
 \end{itemize}

The following example shows that -- even when the poset is a Boolean algebra or a Riesz space -- $O_2$-convergence does not imply $O_1$-convergence.  This example is credited to D. Fremlin \cite[Ex. 2, pg. 140]{ShaeferBLattices}.  We include it for completeness.

\begin{exa}[$O_2$-convergence is not the same as $O_1$-convergence]\label{exo1o2}
 Let $X$ be an uncountable set and $\mathcal{A}$  the algebra of all finite and cofinite subsets of $X$ (partially ordered by inclusion). Let $x_1,x_2, x_3,\dots$ be a sequence of distinct elements of $X$ and $A_n:=\{x_n\}$. Then the sequence $(A_n)_{n\in\mathbb N}$ O$_2$-converges in $\mathcal{A}$ to $\emptyset$ since the filtered family of all cofinite subsets of $X$ has infimum $\emptyset$, but is not O$_1$-convergent since any O$_1$-convergent sequence in $\mathcal{A}$ is eventually constant.

Let $S(\mathcal{A}):= span\{\chi_A :A\in\mathcal{A}\}$ be the Riesz space of $\mathbb{R}$-valued $\mathcal{A}$-simple functions. Then the sequence $(\chi_{A_n})_{n\in\mathbb N}$  O$_2$-converges  to $0$, but is not O$_1$-convergent in $S(\mathcal{A})$.
\end{exa}

 A subset $X$ of $P$ is said to be \emph{$O_i$-closed} if  there is no net in $X$ that is $O_i$-converging to a point outside of $X$.    In  \cite[Theorem 5]{AbChWe} it was shown that $X$ is $O_1$-closed iff it is $O_2$-closed.  The collection of all $O_1$-closed (= $O_2$-closed) subsets of $P$ forms the\emph{ order topology $\tau_O(P)$}. It is easily seen that $[a,b]$ is $O_1$-closed for every $a\le b \in P$.

Motivated by the notion of \emph{unbounded order convergence} (uo-convergence for short), studied for Riesz spaces in \cite{GaTrXa2017,GAOLEU17,GaXa,Ka99},  let us propose the following definition.

\begin{defi}
Let  $P$ be a poset and let $F\subseteq  P^P$.  Let $i\in\{1,2\}$.
\begin{enumerate}[{\rm(i)}]
  \item The net $(x_\gamma)_{\gamma\in\Gamma}$ is said to \emph{$FO_i$-converge to $x$ in $P$}  if $(f(x_\gamma))_{\gamma\in\Gamma}$ is $O_i$-convergent to $f(x)$ for every $f\in F$.
  \item A subset $X\subseteq P$ is said to be \emph{$FO_i$-closed} if there is no net in $X$ that is $FO_i$-converging to a point outside of $X$.
\end{enumerate}
\end{defi}

Clearly, $x_\gamma\converges{$FO_1$} x\ \Rightarrow\  x_\gamma\converges{$FO_2$}x$,  for every net $(x_\gamma)_{\gamma\in \Gamma}$ and $x$ in the poset $P$, and $F\subseteq P^P$.
Although $FO_1$-convergence is different from $FO_2$-convergence,  we shall prove the following result that will  play an important role in order to  answer  \cite[Question 5.10]{T}.

\begin{theo}\label{t3}  Let $P$ be a poset and  $F\subseteq P^P$.  Let $(x_\gamma)_{\gamma\in\Gamma}$ be a net in $P$ that  $F{O}_2$-converges to $x\in P$.  Then $(x_\gamma)_{\gamma\in \Gamma}$ has a subnet that  $FO_1$-converges to $x$.  \\
In particular, if a net $(x_\gamma)_{\gamma\in\Gamma}$ of $P$ is ${O}_2$-convergent to $x\in P$,  then $(x_\gamma)_{\gamma\in \Gamma}$ has a subnet that  $O_1$-converges to $x$.
\end{theo}
\begin{proof}
  For every $f\in F$ let $M_f$ and $N_f$ be the directed, and respectively, the filtered subsets of $P$  `witnessing' the $O_2$-convergence of $(f(x_\gamma))_{\gamma\in\Gamma}$ to $f(x)$, i.e. $M_f$ and $N_f$ satisfy: {\rm(i)}~$\sup M_f=\inf N_f=f(x)$, and {\rm(ii)}~$(f(x_\gamma))_{\gamma\in\Gamma}$ is eventually in $[m,n]$ for every $m\in M_f$ and $n\in N_f$.

  For every finite subset $F_0$ of $F$, let $M_{F_0}$ denote the set of all $g\in P^{F_0}$ satisfying $g(f)\in M_f$ for every $f\in F_0$; and similarly, let $N_{F_0}$ denote the set of all $h\in P^{F_0}$ satisfying $h(f)\in N_f$ for every $f\in F_0$.  Define
  \[\Upsilon:=\{(F_0, g, h)\,:\,F_0\subseteq F,\, |F_0|<\infty,\, g\in M_{F_0},\, h\in N_{F_0}\}.\]
  It can easily be verified that the relation
  \[(F_0,g,h)\le (F_0',g',h')\  \leftrightarrow\ F_0\subseteq F_0'\  \mbox{and}\ g(f)\le g'(f),\ h(f)\ge h'(f)\ \forall f\in F_0\]
  induces a partial order on $\Upsilon$ and that $(\Upsilon,\le)$ is a directed set.   Moreover, for every $(F_0,g,h)\in\Upsilon$, there exists $\gamma(F_0,g,h)\in \Gamma$ such that $f(x_\gamma)\in [g(f), h(f)]$ for every $f\in F_0$ and for every $\gamma\ge \gamma(F_0,g,h)$.

  The product $\Upsilon\times \Gamma$ of the two directed sets $\Upsilon$ and $\Gamma$ is again a directed set and the subset $\Lambda\subseteq\Upsilon\times \Gamma$ consisting of all $(F_0,g,h,\gamma):=\bigl((F_0,g,h),\gamma\bigr)\in \Upsilon\times \Gamma$ satisfying $\gamma\ge \gamma(F_0,g,h)$ is cofinal in $\Upsilon\times\Gamma$.  Hence, $\Lambda$ is a directed set when equipped with the partial order inherited from $\Upsilon\times\Gamma$.  One can easily verify that the projection $\varphi:(F_0,g,h,\gamma)\mapsto \gamma$ is isotone and final; i.e. the net $\bigl(x_{\varphi(\lambda)}\bigr)_{\lambda\in\Lambda}$ is a subnet of $(x_\gamma)_{\gamma\in\Gamma}$.

   We show that the net $(x_{\varphi(\lambda)})_{\lambda\in\Lambda}$ is $FO_1$-convergent to $x$.  Fix an arbitrary $f\in F$.    Let $\Lambda_f:=\{(F_0,g,h,\gamma)\in\Lambda:f\in F_0\}$ and for every $(F_0,g,h,\gamma)\in \Lambda_f$ define $m(F_0,g,h,\gamma):=g(f)$ and $n(F_0,g,h,\gamma):=h(f)$.  By construction,
the net $(m(\lambda))_{\lambda\in\Lambda_f}$ is increasing, the net $(n(\lambda))_{\lambda\in\Lambda_f}$ is decreasing and $m(\lambda)\leq f(x(\lambda))\leq n(\lambda)$ for all $\lambda\in\Lambda_f$. Moreover,
   \begin{gather*}
 \sup\{m(\lambda):\lambda\in \Lambda_f\}=\sup M_f =f(x),\\
 \inf\{n(\lambda):\lambda\in \Lambda_f\}=\inf N_f =f(x).
 \end{gather*}
Therefore $\bigl(f(x_{\varphi(\lambda)})\bigr)_{\lambda\in \Lambda_f}$ is $O_1$-convergent to $f(x)$.
Since $\{\lambda\in\Lambda : \lambda\geq\lambda_0\}=\{\lambda\in\Lambda_f : \lambda\geq\lambda_0\}$ for $\lambda_0\in \Lambda_f$, also $\bigl(f(x_{\varphi(\lambda)})\bigr)_{\lambda\in \Lambda}$ is $O_1$-convergent to $f(x)$.
 This shows that  $(x_{\varphi(\lambda)})_{\lambda\in\Lambda}$ is $FO_1$-convergent to $x$.
 In particular, by setting $F$ equal to the singleton set containing just the identity function on $P$, one obtains that every net of $P$ that is $O_2$-convergent to some point admits a subnet that is $O_1$-convergent to the same point.
  \end{proof}

  \begin{cor}\label{c3} Let $P$ be a poset and  $F\subseteq P^P$.
  \begin{enumerate}[{\rm(i)}]
    \item For $X\subseteq P$, denote by $\bar{X}^{FO_i}$ the set of $x\in P$ such that there is a net in $X$ which $FO_i$-converges to $x$. Then $\bar{X}^{FO_1}=\bar{X}^{FO_2}$.
    \item A subset $X\subseteq P$ is $FO_1$-closed iff it is $FO_2$-closed.  (In particular, by setting $F$ equal to the singleton set containing the identity map on $P$, one recovers \cite[Theorem 6]{AbChWe}.)
    \item The collection of all $FO_1$-closed (=$FO_2$-closed) subsets of $P$ forms a topology on $L$.  Denote this topology by $\tau_{FO}(P)$.
    \item Let $(Y,\tau)$ be a topological space.  For every $\varphi:P\to Y$  the following assertions are equivalent:
    \begin{enumerate}[{\rm(a)}]
      \item $\varphi$ is $\tau_{FO}(P)-\mathcal \tau$-continuous;
      \item If $(x_\gamma)_{\gamma\in \Gamma}$ is $FO_2$-convergent to $x$ in $P$, then $(\varphi(x_\gamma))_{\gamma\in\Gamma}$ is convergent to $\varphi(x)$ w.r.t. $\tau$;
      \item If $(x_\gamma)_{\gamma\in \Gamma}$ is $FO_1$-convergent to $x$ in $P$, then
      $(\varphi(x_\gamma))_{\gamma\in\Gamma}$ is convergent to $\varphi(x)$ w.r.t. $\tau$.
      \end{enumerate}
  \end{enumerate}
\end{cor}
\begin{proof}
{\rm(i)}~is a direct consequence of Theorem \ref{t3}. {\rm(ii)} follows from {\rm(i)}.  The proof of {\rm(iii)} is routine.  The implications {\rm(a)}$\Rightarrow${\rm(b)}$\Rightarrow${\rm(c)} of {\rm(iv)} are trivial.  To show {\rm(c)}$\Rightarrow${\rm(a)} observe that if $X\subseteq Y$ is $\mathcal T$-closed, and $(x_\gamma)_{\gamma\in\Gamma}$ is a net in $\varphi^{-1}X$ that is $FO_1$-convergent to $x$, then $\varphi(x)\in X$ by {\rm(c)}.  So, $x\in\varphi^{-1}X$, i.e. $\varphi^{-1}X$ is $FO_1$-closed (i.e. $\tau_{FO}(P)$-closed).
\end{proof}

The notion of unbounded order convergence on Riesz spaces has received a considerable attention.  For general results on this topic the reader is referred to \cite{GaTrXa2017} and \cite{GaXa}.  Let us recall that the notion of unbounded order convergence  is an abstraction of almost everywhere convergence in function spaces and originally goes back to \cite{NAK48}.

Unbounded order convergence can be generalized for lattices: Let $L$ be a lattice and define  $f_{s,t}$ on $L$ by  $f_{s,t}(x):=(x\wedge t)\vee s$  for  $s,t,x\in L$. If $F=\{f_{s,t}:s,t\in L\;,\; s\leq t\}$, we write $x_\gamma\converges{\text{$\mathfrak{u}O_i$}} x$ instead of $x_\gamma\converges{\text{$FO_i$}} x$ and $\mathfrak{u}O$-closed instead of $FO_1$-closed. In view of Proposition \ref{p3}, $\mathfrak{u}O_2$-convergence in Riesz spaces coincides with unbounded order convergence as defined e.g. in \cite{GaTrXa2017}.

\begin{cor}\label{c5} Let $\tau$ be a topology on a lattice $L$. Then the following conditions are equivalent:
\begin{itemize}
  \item $x_\gamma\converges{\text{$\mathfrak{u}O_1$}} x\ \Rightarrow\ x_\gamma\converges{\text{$\tau$}} x$,
  \item $x_\gamma\converges{\text{$\mathfrak{u}O_2$}} x\ \Rightarrow\ x_\gamma\converges{\text{$\tau$}} x$.
\end{itemize}
\end{cor}

The fact that the two conditions of Corollary \ref{c5} (called in \cite{T} the $\mathfrak{u}O_1$-Lebesgue property and the $\mathfrak{u}O_2$-Lebesgue property) are equivalent, answers \cite[Question  5.10]{T}.

Let $A$ be a solid subgroup of a commutative $\ell$-group $(G,+)$.  For $i\in\{1,2\}$, the net $(x_\gamma)_{\gamma\in \Gamma}$ is said to $\mathfrak{u}_AO_i$-converge to $x$ in $G$ if $|x_\gamma-x|\wedge a\converges{\text{$O_i$}}x$ for every $a\in A_+$.


The proof of the following lemma is basically contained in the proof of \cite[Proposition 7.2]{Papa1964}.
It is also used in the proof of Theorem \ref{cf}.

\begin{lem}\label{l5} Let $(G,+)$ be a commutative $\ell$-group. For every $s,x,y\in G$ and $a\in G_+$
\[|f_{s,s+a}(x)-f_{s,s+a}(y)|\,\le\, |x-y|\wedge a\,=\, |f_{y-a,y}(x)-f_{y-a,y}(y)|+|f_{y,y+a}(x)-f_{y,y+a}(y)|.\]
\end{lem}
\begin{proof}
To prove the inequality on the left-hand side observe that $|f_{s,s+a}(x)-f_{s,s+a}(y)|\le a$ follows since $f_{s,s+a}(x), f_{s,s+a}(y)\in [s,s+a]$, and the inequality $|f_{s,s+a}(x)-f_{s,s+a}(y)|\le |x-y|$ follows by \cite[p.~296]{Birkhoff1967ThirdEd}.

 Let us prove the equality on the right-hand side.  First observe that
 \begin{align*}
 f_{y,y+a}(x)-f_{y,y+a}(y)=&(x \wedge (y+a)) \vee y\, -\, (y \wedge (y+a)) \vee y,\\
 =&(x \wedge (y+a)) \vee y \,-\, y\\
 =&((x-y)\wedge a)\vee 0 \\
 =&(x-y)^{+} \wedge a,
 \end{align*}
and
\begin{align*}
f_{y-a,y}(x)-f_{y-a,y}(y)=&(x\wedge y)\vee(y-a)\,-\,y\\
=&((x-y)\wedge 0)\vee(-a)\\
=&(-(x-y)^-)\vee(-a)\\
=&-((x-y)^-\wedge a).
\end{align*}
Therefore, since $(x-y)^{+}\wedge a$ and $(x-y)^{-}\wedge a$ are disjoint, we obtain
\begin{align*}
|f_{y,y+a}(x)-f_{y,y+a}(y)|+|f_{y-a,y}(x)-f_{y-a,y}(y)|=&((x-y)^{+} \wedge a)\,+\,((x-y)^-\wedge a)\\
=&((x-y)^{+} \wedge a)\,\vee\,((x-y)^-\wedge a)\\
=&|x-y|\wedge a.
\end{align*}
  \end{proof}

Let us recall that in an $\ell$-group $O_i$-convergence is: {\rm(i)}~additive, i.e. if $x_\gamma\converges{\text{$O_i$}}x$ and $y_\gamma\converges{\text{$O_i$}}y$, then $x_\gamma+y_\gamma\converges{\text{$O_i$}}x+y$; and {\rm(ii)}~`commutes with taking inverses', i.e.  $x_\gamma\converges{\text{$O_i$}}x$ iff $-x_\gamma\converges{\text{$O_i$}}-x$. Making use of the decomposition $x=x^+-x^-$, one can easily deduce that $x_\gamma\converges{\text{$O_i$}}0$ iff $|x_\gamma|\converges{\text{$O_i$}}0$.  Hence, the following proposition immediately follows by Lemma \ref{l5}.

\begin{prop}\label{p3} Let $(G,+,\tau)$ be a commutative $\ell$-group and $A$ a solid subgroup of $G$.   Let
\[F:=\{f_{s,t}\,:\,s,t\in G,\, t-s\in A_+\}.\]
 Then $x_\gamma\converges{\text{$\mathfrak{u}_AO_i$}} x$ iff $x_\gamma\converges{\text{$FO_i$}}x$, for every net $(x_\gamma)_{\gamma\in\Gamma}$ and $x$ in $G$.
\end{prop}

Proposition \ref{p3} shows that, for $A=G$,  $\mathfrak{u}_GO_i$-convergence coincides with $\mathfrak{u}O_i$-convergence.
 For $A=G$ and $i=2$, Proposition \ref{p3} coincides with \cite[Proposition 7.2]{Papa1964}.

\begin{cor}
  Let $(G,+)$ be a commutative $\ell$-group .  Let $(x_\gamma)_{\gamma\in\Gamma}$ be a net in $G$, $x\in G$ and $x_\gamma\converges{\text{$\mathfrak{u}O_2$}}x$. Then $(x_\gamma)_{\gamma\in\Gamma}$ has a subnet that  $\mathfrak{u}O_1$-converges to $x$.
\end{cor}
\begin{proof}
This follows at once by Proposition \ref{p3} and Theorem \ref{t3}.
\end{proof}

\begin{cor}\label{c4} Let $(G,+)$ be a commutative $\ell$-group  and $A$ a solid subgroup of $G$.  Let $(Y,\tau)$ be a topological space and $\varphi$ a function from $G$ into $Y$.  Then the following conditions for $\varphi$ are equivalent:
\begin{itemize}
  \item $x_\gamma\converges{\text{$\mathfrak{u}_AO_1$}} x\ \Rightarrow\ \varphi(x_\gamma)\converges{\text{$\tau$}} \varphi(x)$,
  \item $x_\gamma\converges{\text{$\mathfrak{u}_AO_2$}} x\ \Rightarrow\ \varphi(x_\gamma)\converges{\text{$\tau$}}\varphi(x)$.
\end{itemize}
\end{cor}
\begin{proof}
  This follows by Proposition \ref{p3} and Corollary \ref{c3}
\end{proof}

In view of Proposition \ref{p3}, Corollary \ref{c3}(i) answers \cite[Question 18.51]{TaylorThesis}.

\section{Lattice uniformities -- setting up the basics}
Let $L$ be a lattice and $\Delta:=\{(x,x):x\in L\}$ the diagonal of $L\times L$.
A subset $C$ of $L$ is called \textit{convex} if $[a, b]\subseteq C$ whenever $a, b\in C$ and $a\leq b.$
A \textit{lattice uniformity} is a uniformity on a lattice making the lattice operations $\vee$
and $\wedge$ uniformly continuous. A \textit{uniform lattice} is a lattice endowed with a lattice uniformity.
One easily verifies:

\begin{prop} [\cite{W91}, Proposition 1.1.2] \label{dia}
Let  $u$ be a uniformity on the lattice   $L$.  Then the following statements are equivalent:
\begin{enumerate}[{\rm(i)}]
\item $u$  is a lattice uniformity;
\item for every $U\in u$ there exists  $V\in u$ satisfying  $V\vee V\subseteq U$ and $V\wedge V\subseteq U$;
\item for every $U\in u$ there exists  $V\in u$ satisfying $V\vee\Delta\subseteq U$ and $V\wedge\Delta\subseteq U$.
\end{enumerate}
\end{prop}

It follows that the supremum of lattice uniformities on $L$ (built in the complete lattice of all uniformities on $L$) is also a lattice uniformity.  We remark that -- although it is more delicate to see -- the analogous statement holds true also for the infimum of lattice uniformities, (see \cite[Remark 2.6]{MR2377916}).

The topology induced by a lattice uniformity $u$ on $L$, the $u$-topology, is a \textit{locally convex
lattice topology}, i.e. $\vee$ and $\wedge$ are continuous and every point of $L$ has a neighbourhood base consisting of convex sets.

As mentioned in the introduction, uniform lattices generalize topological Boolean rings on one-hand and locally solid $\ell$-groups (and therefore locally solid Riesz spaces) on the other.  More specifically, we recall that (see \cite[Proposition 1.1.7 \& 1.1.8]{W91}) the uniformity induced by
\begin{itemize}
\item an FN-topology (i.e. a locally convex group topology) on a Boolean ring, or
\item a locally solid group topology\footnote{A subset $A$ of an $\ell$-group $G$ is solid if for every for every $a\in A$ and $x\in G$ satisfying $|x|\le |a|$, it holds that $x\in A$.  A group topology on an $\ell$-group $G$ is locally solid if $0$ has a neighbourhood base consisting of solid sets. } on an $\ell$-group,
\end{itemize}
is a lattice uniformity.

For the lattice uniformity $u$ we let $N(u):=\cap_{U\in u}U$.  We note that $u$ is Hausdorff iff $N(u)=\Delta$. We recall  \cite[Proposition 1.1.3]{W91} that if $u$ is a lattice uniformity on $L$,  every $U\in u$ contains a $V\in u$ such that the rectangle $[a\wedge b,a\vee b]^2\subseteq V$ for all pairs $(a,b)\in V$.

\begin{prop}\label{p2}
 Let $u$ and $v$ be two lattice uniformities on $L$  that agree on every order bounded subset of  $L$.    Then $N(u)=N(v)$. In particular, $u$ is Hausdorff iff $v$ is Hausdorff.
 \end{prop}
 \begin{proof}
 If $(a,b)\in N(u)$, then the rectangle $[a\wedge b,a\vee b]^2\subseteq N(u)$ and so, since the uniformity $\restr{u}{[a\wedge b,a\vee b]}$ induced by $u$ on $[a\wedge b,a\vee b]$ agrees with $\restr{v}{[a\wedge b,a\vee b]}$, we get $[a\wedge b,a\vee b]^2\subseteq N(v)$, i.e. $(a,b)\in N(v)$.  This shows $N(u)\subseteq N(v)$ and therefore $N(u)=N(v)$  by symmetry.
 \end{proof}

 A lattice uniformity $u$ on $L$ is said to be:
\begin{itemize}
  \item  \emph{order continuous (o.c.)} if $(x_\gamma)_{\gamma\in\Gamma}$ converges to $x$ in $(L,u)$ whenever $(x_\gamma)_{\gamma\in\Gamma}$ is a monotone net in $L$ order converging to $x$;
  \item  \emph{exhaustive} if every monotone net in $L$ is Cauchy;
   \item  \emph{locally exhaustive} if $\restr{u}{[a,b]}$ is exhaustive for every $a\le b$ in $L$.
 \end{itemize}

It is easy to see that a lattice uniformity $u$ is o.c. iff  the $u$-topology  is coarser than the order topology $\tau_O(L)$: Let $u$ be o.c. and $C$ be closed in $(L,u)$. If $(x_\gamma)_{\gamma\in\Gamma}$ is a net in $C$, $x\in L$ and $(y_\gamma)_{\gamma\in\Gamma}$, $(z_\gamma)_{\gamma\in\Gamma}$ are nets in $L$ such that $y_\gamma\uparrow x$, $z_\gamma\downarrow x$ and $y_\gamma\leq x_\gamma\leq z_\gamma$ for all $\gamma\in\Gamma$, then both, $(y_\gamma)$ and $(z_\gamma)$, converge to $x$ w.r.t. $u$, therefore $(x_\gamma)$ converges to $x$ w.r.t. $u$ since the $u$-topology is locally convex, hence $x\in C$, i.e. $C$ is order-closed. The other implication, that $u$ is o.c. if the $u$-topology is coarser than $\tau_O(L)$, is obvious.

By \cite[Proposition 6.1]{W93}, $u$ is exhaustive iff every monotone sequence in $L$ is Cauchy.

\begin{theo}[\cite{W93}, Proposition 6.3] \label{t1} 
Let $(L,u)$ be a Hausdorff uniform lattice. Then the following conditions are equivalent:
\begin{enumerate}[{\rm(i)}]
\item  $(L,u)$ is complete (as a uniform space) and  $u$ is exhaustive;
\item $(L,\le)$ is a complete lattice and $u$ is o.c..
\end{enumerate}
\end{theo}

\begin{prop}[\cite{W93}, Proposition 6.14] \label{exh}
Let $S$ be a dense sublattice of a uniform lattice $(L,u)$. If  $u|_S$  is (locally) exhaustive, then $u$ is (locally) exhaustive, too.
\end{prop}

We recall that every Hausdorff uniform lattice $(L,u)$  is a sublattice and a dense subspace of a Hausdorff uniform lattice $(\tilde L,\tilde u)$, which is complete as a uniform space (see \cite[Proposition 1.3.1]{W91}. $(\tilde L,\tilde u)$ is called the completion of $(L,u)$.

From Theorem \ref{t1} and Proposition \ref{exh} immediately follows

\begin{cor}\label{com}
Let $(L,u)$ be a (locally) exhaustive Hausdorff uniform lattice and $(\tilde L,\tilde u)$ its completion. Then $(\tilde{L},\leq)$ is a (Dedekind) complete lattice and $\tilde{u}$ is o.c..
\end{cor}

In Theorem \ref{W93C7.2.4} we need for a uniform lattice $(L,u)$ the following Condition (C) stronger than order continuity.

 \emph{Condition (C)}: For every increasing (and resp. decreasing) order bounded net $(x_\gamma)_{\gamma\in\Gamma}$ in $L$ and every $U\in u$, there exists an upper bound (resp. lower bound) $x\in L$ of $(x_\gamma)_{\gamma\in\Gamma}$, such that $(x_\gamma,x)\in U$ eventually.

\begin{prop}[\cite{W93}, Proposition 7.1.2] \label{7.1.2} 
\begin{enumerate}[{\rm(i)}]
\item Any lattice uniformity satisfying Condition (C) is o.c. and locally exhaustive.
\item If $u$ is a lattice uniformity on a Dedekind complete lattice, then $u$ satisfies Condition (C) iff $u$ is  o.c..
\end{enumerate}
\end{prop}

 \begin{theo}[\cite{W93}, Corollary 7.2.4]\label{W93C7.2.4}
 Let $u$ and $v$ be exhaustive lattice uniformities on the lattice $L$ satisfying the Condition (C).  If $N(u)=N(v)$, then the $u$-topology coincides with the $v$-topology.
 \end{theo}

Let us briefly consider locally solid $\ell$-groups (in particular, locally solid Riesz spaces). Following \cite{AB} (for locally solid Riesz spaces) we call a locally solid group topology a \emph{Lebesgue topology} (resp. \emph{pre-Lebesgue topology}) when the induced uniformity is o.c. (resp. locally exhaustive).
In \cite[Proposition 7.1.5]{W93} it is shown that the uniformity induced by a locally solid group topology on an Archimedean $\ell$-group satisfies Condition (C) iff it is o.c.. Therefore, by Proposition \ref{7.1.2} (a), if an Archimedean locally solid $\ell$-group is Lebesgue, then it is pre-Lebesgue; this generalizes \cite[Theorem 3.23 on pg 87]{AB} where this implication is given for Archimedean locally solid Riesz spaces.

The `uniqueness' result in Theorem \ref{W93C7.2.4} implies that if $\tau$ and $\sigma$ are Hausdorff Lebesgue locally solid group topologies on an Archimedean $\ell$-group $G$, then $\sigma$ and $\tau$ induce the same topology on the order bounded subsets of $G$. In particular, for locally solid Riesz spaces, this gives \cite[Theorem 4.22 on pg 106]{AB}.

Sometimes, for example to apply Theorem \ref{t1}, it is convenient to consider the Hausdorff uniform lattice associated with a uniform lattice $(L,u)$.  Let us briefly outline this factorisation procedure.

\begin{prop}[\cite{W91}, Proposition 1.2.4]\label{q}
Let $(L,u)$ be a uniform lattice. Define the equivalence relation $\sim_u$ by setting $a\sim_u b$ iff $(a,b)\in N(u)$. Let $\hat{a}$ denote the equivalence class containing the element $a\in L$.  Then the  operations $\hat a\vee \hat b:=\widehat{a\vee b}$ and $\hat a\wedge \hat b:=\widehat{a\wedge b}$ turn the quotient $\hat{L}:= L/\sss{\sim_u}$ into a lattice, and the family $\widehat u:=\{\widehat U: U\in u\}$, where $\widehat U:=\{(\hat a,\hat b):(a,b)\in U\}$, is a Hausdorff lattice uniformity on $\hat{L}$.  For any closed subset $U$ of $(L,u)^2$ we have $(a,b)\in U$ iff $(\hat a,\hat b)\in \hat U$.
\end{prop}

Lattice uniformities can be generated by systems of particular semimetrics\footnote{Semimetrics are assumed to be real-valued.}. If $q>1$ and $D$ is a system of semimetrics satisfying
\begin{equation*}
d(x\vee z,y\vee z)\le q\cdot d(x,y)\quad\text{and}\quad d(x\wedge z, y\wedge z)\le q\cdot d(x,y)\quad\text{for every $x,y,z\in L$ and $d\in D$,}
\end{equation*}
then the uniformity generated by $D$ on $L$ is a lattice uniformity. Vice versa we have:

\begin{prop}[\cite{W3}, Corollaries 1.4 and 1.6] \label{semi}
Let $(L,u)$ be a uniform lattice and $q>1$. Then $u$ is generated by a system $D$ of semimetrics satisfying
\begin{equation*}
d(x\vee z,y\vee z)\le q\cdot d(x,y)\quad\text{and}\quad d(x\wedge z, y\wedge z)\le q\cdot d(x,y)\quad\text{for every $x,y,z\in L$ and $d\in D$.}
\end{equation*}

If $L$ is distributive, then one can here choose also $q=1$.
\end{prop}

Let $(L,u)$ be a uniform lattice and $F$ a set of lattice homomorphisms $f:L\rightarrow (L,u)$. The sets $U_f:=\{(x,y)\in L^2: (f(x),f(y))\in U\}$, where $U\in u$ and $f\in F$, form a subbase of the initial uniformity of $F$, i.e. of the coarsest uniformity on $L$ making $f$ uniformly continuous for all $f\in F$ (see \cite[Proposition 4 of II.2.3]{Bourbaki}).
Using Proposition \ref{dia}, one sees that this uniformity is a lattice uniformity: If $U,V\in u$ with $V\vee \Delta\subseteq U$ and $V\wedge \Delta\subseteq U$, then $V_f\vee \Delta\subseteq U_f$ and $V_f\wedge \Delta\subseteq U_f$ for all $f\in F$.

We are here interested in the initial uniformity w.r.t. functions of the type  $f_{a,b}(x):=(x\wedge b)\vee a$ or  $g_{a,b}(x):=(x\vee a)\wedge b$. Proposition \ref{d} {\rm (a)} is the reason that later on we only consider distributive lattices. Proposition \ref{d} {(b)} shows that it is indifferent whether we study  the initial uniformity w.r.t. functions of the type  $f_{a,b}$ or of the type $g_{a,b}$; moreover, we can hereby always assume that $a\leq b$.

\begin{prop}\label{d}
For a lattice $L$ the following statements are equivalent:
\begin{enumerate}[{\rm(i)}]
\item  $L$ is distributive.
\item  $f_{a,b}$ is a lattice homomorphism for all $a,b\in L$.
\item $g_{a,b}$ is a lattice homomorphism for all $a,b\in L$.
\end{enumerate}
(b) If $L$ is distributive, then
 $f_{a,b}=f_{a,a\vee b}=g_{a,a\vee b}$ and  $g_{a,b}=g_{a\wedge b,b}=f_{a\wedge b,b}$
  for every $a,b\in L$.
\end{prop}

We omit the proof which is routine.

\vskip 13pt
\section{ The uniformity $u^*$ inducing `unbounded $u$-convergence'.}

In this section, let $u$ be a lattice uniformity on a distributive lattice $L$.

  For $a,b\in L$, let $u_{a,b}$ be the initial uniformity of $f_{a,b}$. Since $f_{a,b}:(L,u)\rightarrow (L,u)$ is uniformly continuous,  $u_{a,b}\subseteq u$ and, as explained before, $u_{a,b}$ is a lattice uniformity since $L$ is assumed to be distributive. Moreover, $v_{a,b}\subseteq u_{a,b}$ if $v$ is a lattice uniformity on $L$ coarser than $u$.

\begin{prop}\label{p1}  Let $a,b,c,d \in L$.  Then
\[u_{a\vee(b\wedge c),\,b\wedge d}\subseteq  (u_{a,b})_{c,d}.\]
In particular:
\begin{enumerate}[{\rm(i)}]
\item If $a\le c\le d\le b$, then $u_{c,d}\subseteq u_{a,b}$, and
\item  $\left(u_{a,b}\right)_{a,b}=u_{a,b}$.
\end{enumerate}
\end{prop}

\begin{proof}
First observe that for $x\in L$
\begin{align*}
f_{a,b}\circ f_{c,d}(x) &=(((x\wedge d)\vee c)\wedge b)\vee a \\
   &=(x\wedge d\wedge b)\vee(c\wedge b)\vee a \\
   &=(x\wedge (d\wedge b))\vee (a\vee (b\wedge c))=f_{a\vee (b\wedge c),\, b\wedge d}(x),
\end{align*}
i.e.  $f_{a,b}\circ f_{c,d}= f_{a\vee (b\wedge c),\, b\wedge d}$.
The functions $f_{c,d}:\left (L,(u_{a,b})_{c,d}\right )\to (L,u_{a,b})$ and $f_{a,b}:(L,u_{a,b})\to(L,u)$ are uniformly continuous, and therefore so is the function
\[f_{a\vee (b\wedge c),\, b\wedge d}=f_{a,b}\circ f_{c,d}:\left(L,\left(u_{a,b}\right)_{c,d}\right)\to (L,u).\] This shows that $u_{a\vee(b\wedge c),\,b\wedge d}\subseteq (u_{a,b})_{c,d}$.

In particular, if $a\le c\le d\le b$, then $a\vee(b\wedge c)=c$ and $b\wedge d=d$, and therefore $u_{c,d}\subseteq \left(u_{a,b}\right)_{c,d}\subseteq u_{a,b}$, and  by setting $a=c$ and $b=d$ one also obtains $u_{a,b} = (u_{a,b})_{a,b}$.
\end{proof}

As observed before, it is enough to consider the uniformities $u_{a,b}$ for $a\leq b$. Define $J(L):=\{(a,b)\in L^2: a\leq b\}$.
  If $\emptyset \neq J\subseteq J(L)$ then the initial uniformity $u_J$ of $\{f_{a,b}: (a,b)\in J\}$ is the uniformity having as a subbase the sets of the form
\[U_{a,b}:=\{(x,y)\in L^2:(f_{a,b}(x),f_{a,b}(y))\in U\}\quad ((a,b)\in J,\ U\in u).\]
Note that $u_J$ is the supremum of the uniformities $u_{a,b}$ where $(a,b)\in J$, and thus $u_J$ is again a lattice uniformity.  Let $u^*:=u_{J(L)}$.

\begin{prop}\label{uj}
Let $J$ and $J'$ be non-empty subsets of $J(L)$.
\begin{enumerate}[{\rm(i)}]
\item $v_J\subseteq u_J\subseteq u$ for every lattice uniformity $v$ coarser than $u$.
\item If  for every $(a,b)\in J$ there exists $(a',b')\in J'$ satisfying $a'\leq a\leq b\leq b'$, then $u_J\subseteq u_{J'}$ and $(u_J)_{J'}=(u_{J'})_J=u_J$.
\item $u_J$ is the weakest lattice uniformity which agrees with $u$ on $[a,b]$ for all $(a,b)\in J$.
\item The $u_J$-topology is the weakest lattice topology which agrees with the $u$-topology on $[a,b]$ for all $(a,b)\in J$.
\end{enumerate}
\end{prop}

\begin{proof}
(i)~is obvious.

(ii)~The first assertion follows  by {\rm(i)} of Proposition \ref{p1} and by the fact that $u_J$ is the supremum of the uniformities $u_{a,b}$ for all $(a,b)\in J$.

For the second assertion, we first note that $(u_J)_{J'}\subseteq u_J$ by (i).  Using {\rm (i)} and  Proposition \ref{p1}, we get $u_{a,b}=(u_{a,b})_{a,b}\subseteq (u_J)_{a,b}\subseteq (u_{J})_{J'}$. Hence $u_J\subseteq(u_J)_{J'}$ since $u_J$ is the supremum of the uniformities $u_{a,b}$, $(a,b)\in J$.
Similarly, one can prove $u_J=(u_{J'})_{J}$

{\rm(iii)}~Let us first show that $u_J$ agrees with $u$ on $[a,b]$ for all $(a,b)\in J$. Let $(a,b)\in J$, $U\in u$ and $U_{a,b}:=\{(x,y)\in L^2: (f_{a,b}(x),f_{a,b}(y))\in U\}$. Then $U_{a,b}\cap [a,b]^2=U\cap [a,b]^2$ since $f_{a,b}(x)=x$ for $x\in [a,b]$. Therefore $u_{a,b}$ and $u$ agree on $[a,b]$. It follows that $u_J$ and $u$ agree on $[a,b]$ since $u_{a,b}\subseteq u_J\subseteq u$.  Now, let $v$ be a lattice uniformity on $L$ which agrees with $u$ on $[a,b]$ for all $(a,b)\in J$. We shall show that $u_{a,b}\subseteq v$ for $(a,b)\in J$. Since $v$ is a lattice uniformity, $f_{a,b}:(L,v)\rightarrow (L,v)$ is uniformly continuous. Since $f_{a,b}(L)\subseteq [a,b]$ and $v|_{[a,b]}=u|_{[a,b]}$, also $f_{a,b}:(L,v)\rightarrow (L,u)$ is uniformly continuous. Therefore $u_{a,b}\subseteq v$.

The proof of {\rm(iv)} is similar to that of {\rm(iii)}.
\end{proof}

\begin{cor}\label{weakest}
\begin{enumerate}[{\rm(i)}]
\item $(u^\ast)^\ast=u^\ast$.
\item $u^*$ is the weakest lattice uniformity which agrees with $u$ on any order bounded subset of $L$.
\item The $u^*$-topology is the weakest lattice topology which agrees with the $u$-topology on any order bounded subset of $L$.
\end{enumerate}
\end{cor}

In contrast to Corollary \ref{weakest}, in general, $u^*$ is not the weakest uniformity which agrees with $u$ on every order bounded subset of $L$,
and the $u^*$-topology is not the weakest topology which agrees with the $u$-topology on any order bounded subset of $L$:

\begin{exa}\label{ex-r}
Let $u$ be the usual uniformity on $\mathbb{R}$ induced by the absolute value.
\begin{enumerate}[{\rm(a)}]
\item The infimum $v$ of all uniformities on $\mathbb{R}$ which agree with $u$ on all bounded subsets of $\mathbb{R}$ is the trivial uniformity.
\item  The sets $U_n:=\{(x,y)\in \mathbb{R}^2: |x-y|\leq\frac{1}{n} \text{ or } x,y\geq n \text{ or } x,y\leq -n\}$, $n\in \mathbb{N}$, form a base for $u^*$.
\end{enumerate}
\end{exa}
\begin{proof}
{\rm(a)}~
For $a\in\mathbb{R}$, and $n\in\mathbb N$ let $V_{n,a}$ be the subset of $\mathbb R^2$ defined by
 \[(x,y)\in V_{n,a}\ \Longleftrightarrow\ \begin{cases}
                                           |x-y|\leq\frac{1}{n},\ \mbox{or} \\
                                            |x|, |y|\geq n,\ \mbox{or}\\
                                           |x|\geq n, |y-a|\leq \frac{1}{n},\ \mbox{or}\\
                                           |y|\geq n, |x-a|\leq \frac{1}{n}\,.
                                           \end{cases}\]
                                           The family $\{V_{n,a}:n\in\mathbb N\}$ form a base for a uniformity $v_a$ on $\mathbb{R}$ which agrees with $u$ on all bounded subsets of $\mathbb{R}$. Then $n\rightarrow a$ $(v_a)$ and therefore $n\rightarrow a$ ($v$) since $v\subseteq v_a$. Let now $V\in v$ and $a,b\in\mathbb{R}$. Then $n\rightarrow a$ ($v$) and $n\rightarrow b$ ($v$), hence $(a,b)\in V$. This shows $V=\mathbb{R}^2.$

{\rm (b)}~immediately follows from the description of $u^*$ as initial uniformity.
\end{proof}

\begin{rem}
Note that in Example \ref{ex-r}, the $u$-topology coincides with the $u^*$-topology; $u$ is complete whereas $u^*$ is not complete; the completion of $(\mathbb{R},u^*)$ is $[-\infty,+\infty]$ with the usual compact uniformity.
\end{rem}

\begin{prop}\label{pH}
Let $S$ be a sublattice of $L$ and $I:=J(S)=\{(a,b)\in S^2:a\le b\}$.  Then $u_I$ is Hausdorff iff $u$ is Hausdorff and $x=\sup_{s\in S} s\wedge x=\inf_{s\in S}s\vee x$ for every $x\in L$.
\end{prop}
\begin{proof}
Assume that $u_I$ is Hausdorff.  Then also $u$ is  Hausdorff since $u_I\subseteq u$. Let $x\in L$; we show that  $\sup_{s\in S} x\wedge s = x$.  Let $y\le x$ with $x\wedge s\le y$ for every $s\in S$.
Then for every $(a,b)\in I$ we have $x\wedge b=y\wedge b$, therefore $(x\wedge b)\vee a = (y\wedge b)\vee a$ and hence $(x,y)\in N(u_I)=\Delta$, i.e. $x=y$.  The proof of $\inf_{s\in S} x\vee s =x$ is similar.

 Conversely, suppose that $u$ is Hausdorff and that $x=\sup_{s\in S} s\wedge x=\inf_{s\in S}s\vee x$ for every $x\in L$.  Let $(x,y)\in N(u_I)$.  Then, since $u$ is Hausdorff, we get that   $(x\wedge b)\vee a=(y\wedge b)\vee a$ for every $(a,b)\in I$.  Then,
\[(x\wedge b)\vee a=(x\wedge(a\vee b))\vee a=(y\wedge(a\vee b))\vee a=(y\wedge b)\vee a,\]
for arbitrary $a$ and $b$ in $S$, and therefore
\[x=\sup_{b\in S} x\wedge b=\sup_{b\in S}\inf_{a\in S}(x\wedge b)\vee a=\sup_{b\in S}\inf_{a\in S}(y\wedge b)\vee a=\sup_{b\in S}y\wedge b=y.\]
This shows that $u_I$ is Hausdorff.
\end{proof}

\begin{cor}\label{c1} $u$ is Hausdorff iff $u^\ast$ is Hausdorff.
\end{cor}

Corollary \ref{c1} also follows from Corollary \ref{weakest} (ii) and Proposition \ref{p2}.

We now consider the situation when $u$ is the uniformity induced by the topology $\tau$ of a locally solid Riesz space, or more general, of a locally solid commutative $\ell$-group. We will compare the $u_J$-topology and $u^*$-topology, respectively, with $\mathfrak{u}_A\tau$ and  $\mathfrak{u}\tau$ introduced by Taylor \cite{T}. The following theorem was formulated by Taylor \cite[Theorem 2.3 and §9]{T} in the case that $\tau$ is a locally solid linear topology on a Riesz space, but in \cite{TaylorThesis} it is  mentioned a possible generalization for locally solid $\ell$-groups.

\begin{theo}\label{t}
Let $(G,+,\tau)$ be a Hausdorff locally solid commutative $\ell$-group, $\mathfrak{U}$ its $0$-neighbourhood system and $A$ a solid subgroup of $G$. Then the sets
$$\{x\in G: |x|\wedge a\in U\}\; (a\in A_+, U\in \mathfrak{U})$$
form a $0$-neighbourhood base for a locally solid group topology $\mathfrak{u}_A\tau$ on $G$.

A net $(x_\gamma)$ converges to $x$ w.r.t. $\mathfrak{u}_A\tau$ iff $|x_\gamma-x|\wedge a \mathop{\overset{\tau}{\longrightarrow}} 0$ for any $a\in A_+$.

If $(G,\tau)$ is a locally solid Riesz space, then $(G,\mathfrak{u}_A\tau)$ is a locally solid Riesz space, too.
\end{theo}

Following Taylor \cite{T} we write $\mathfrak{u}\tau$ instead of $\mathfrak{u}_A\tau$  if $A=G$.
Convergence w.r.t. $\mathfrak{u}\tau$ is called in \cite{T} \textit{unbounded $\tau$-convergence}. Moreover, in the terminology of Taylor \cite{T}, $\tau$ is called \textit{unbounded} iff $\mathfrak{u}\tau=\tau$.

\begin{theo}\label{cf}
Let $(G,+,\tau)$ be a Hausdorff locally solid commutative $\ell$-group and $A$ a solid subgroup of $G$.  Let $u$ be the uniformity induced by $\tau$ and  $J=\{(a,b)\in G^2: b-a\in A_+\}$. Then the $u_J$-topology coincides with $\mathfrak{u}_A\tau$. In particular, the $u^*$-topology coincides with $\mathfrak{u}\tau$.
\end{theo}
\begin{proof}
Let $\mathfrak{U}$ be  a $0$-neighbourhood base for $(G,\tau)$ consisting of solid sets.
For $s,t \in L$ with $t-s \in A_{+}$ and $U\in \mathfrak U$ let
\begin{equation*}
	\hat{U}(x_{0},s,t) = \{ x\in G :f_{s,t}(x) - f_{s,t}(x_{0}) \in U \}.
\end{equation*}
Then $\{\hat{U}(x_0,s,t)\,:\,U\in\mathfrak U,\, t-s\in A_+\}$ is a neighbourhood subbase of $x_{0}$ w.r.t. the $u_J$-topology. We recall that the family $\{\tilde{U}(x_0,a)\,:\,U\in\mathfrak{U},\,a\in A_+\}$, where
\[\tilde{U}(x_0,a):=\{x\in G:|x-x_0|\wedge a\in U\},\]
is a neighbourhood base of $x_0$ w.r.t. $\mathfrak{u}_A\tau$.

Let $U\in \mathfrak U$ and $s,t\in G$ with $a:=t-s\in A_+$.  We show that $\tilde U(x_{0},a) \subseteq \hat{U}(x_{0},s,t)$. Let $x \in \tilde{U}(x_{0},a)$.  By {\rm(ii)} of Lemma \ref{l5}, one obtains
 \[\vert f_{s,t}(x) - f_{s,t} (x_{0}) \vert  \le |x-x_0|\wedge a\in U,\]
 i.e. $x\in \hat{U}(x_{0},s,t)$.  This shows that  the $u_J\mbox{-topology}$ is coarser than $\mathfrak{u}_A\tau$.

 Let us show the converse. Let $a\in A_+$ and $U\in\mathfrak U$. Choose $V\in\mathfrak U$ with $V+V\subseteq U$, and set $r:=x_0-a$, $s:= x_0$ and $t:= x_0+a$. We show that
$\hat{V}(x_0,s,t)\cap\hat{V}(x_0,r,s)\subseteq  \tilde{U}(x_0,a)$.
For $x\in \hat{V}(x_0,s,t)\cap\hat{V}(x_0,r,s)$ one has by Lemma \ref{l5}
 \[|x-x_0|\wedge a=|f_{r,s}(x)-f_{r,s}(x_0) |\,+\,|f_{s,t}(x)-f_{s,t}(x_0)|\in V+V\subseteq U,\]
 hence $x\in \tilde{U}(x_0,a)$.
 \end{proof}

Under the assumption of Theorem \ref{cf}, let $I:=\{(-a,a): a\in A_+\}$. Then obviously $u_I\subseteq u_J$. The following example shows that the $u_I$-topology can be strictly coarser the $u_J$-topology (=$\mathfrak{u}_A\tau$).

\begin{exa}\label{ex}
Let $G=\mathbb{R}^\mathbb{N}$ and $\tau$ be the locally solid group topology induced by $\Vert(x_n)\Vert:=\sum_{n=1}^\infty|x_n|$. Let $A:=\ell_\infty$, $I:=\{(-a,a): a\in A_+\}$ and $\sigma$ the $u_I$-topology.
Then $\sigma$ is strictly coarser than $\mathfrak{u}_A\tau$.
\end{exa}

\begin{proof}
Let $x=(1,2,3,\dots)$, $e=(1,1,1,\dots)$, $y_n=\frac{1}{n}e$, $x_n=x+y_n$.

We show that $x_n\rightarrow x$ ($\sigma$), but $x_n\nrightarrow x$ ($\mathfrak{u}_A\tau$).

Let $a\in A$, $k\in\mathbb{N}$ with $e\leq a\leq ke$. Then
$|(x_n\wedge a)\vee(-a)-(x\wedge a)\vee(-a)|\leq\frac{1}{n}\chi_{[1,k]}$ where $\chi_{[1,k]}$ denotes the characteristic function of $\{1,2,\dots,k\}$.
Therefore $\Vert(x_n\wedge a)\vee(-a)-(x\wedge a)\vee(-a)\Vert\leq\frac{k}{n}\rightarrow 0$ ($n\rightarrow\infty$).

On the other hand, $|x_n-x|\wedge a=y_n$, $\Vert y_n\Vert=+\infty$ and therefore $(x_n)$ does not converge to $x$ w.r.t. $\mathfrak{u}_A\tau$.
\end{proof}


\section{The uniformity $u^*$ on sublattices}

In this section, let $u$ be a lattice uniformity on a distributive lattice $L$.

Let $D(L)$ denote the set of semimetrics $d$ on $L$ satisfying
\begin{equation*}
d(x\vee z,y\vee z)\le d(x,y)\quad\text{and}\quad d(x\wedge z, y\wedge z)\le d(x,y)\quad\text{for every $x,y,z\in L$.}
\end{equation*}

By Proposition \ref{semi}, $u$ is generated by a subset of $D(L)$, and  therefore by the set $D_u$ of all $d\in D(L)$ which are uniformly continuous w.r.t. $u$.

\begin{rem}\label{r1}
\emph{Let $J\subseteq J(L)$.  For every $(a,b)\in J$ and $d\in D(L)$ denote by $d_{a,b}$ the semimetric on $L$ defined by $d_{a,b}(x,y):=d(f_{a,b}(x),\,f_{a,b}(y))$.   It is easy to see then that the lattice uniformity $u_J$ is generated by the family $\{d_{a,b}:d\in D_u,\,(a,b)\in J\}$.}
\end{rem}

\begin{lem}\label{l2}
Let $d\in D(L)$.
\begin{enumerate}[{\rm(i)}]
\item Let $P_n:L^n\to L$ be given inductively by $P_1(x):=x$ and
\[P_n(x_1,\dots,x_n):=\begin{cases}
                        P_{n-1}(x_1,\dots,x_{n-1})\,\vee\,x_n, & \mbox{or } { } \\
                        P_{n-1}(x_1,\dots,x_{n-1})\,\wedge\,x_n . &
                      \end{cases}\]
                      Then $d(P_n(x_1,\dots,x_n),P_n(y_1,\dots,y_n))\le \sum_{i=1}^{n}d(x_i,y_i)$.
\item $d(f_{a,b}(x),f_{a,b}(y))\le d(f_{c,d}(x),f_{c,d}(y))+2d(a,c)+2d(b,d)$, for every $a,b,c,d,x,y\in L$.
\end{enumerate}
\end{lem}
\begin{proof}
The proof of (i) is by induction.  The assertion is trivially true when $n=1$.  Suppose that
$d(P_{n-1}(x_1,\dots,x_{n-1}),P_{n-1}(y_1,\dots,y_{n-1}))\le \sum_{i=1}^{n-1}d(x_i,y_i)$.  Let $\diamond\in\{\wedge,\vee\}$.  Then,
\begin{align*}
  &d(P_n(x_1,\dots,x_n),P_n(y_1,\dots,y_n))\\
  &=d(P_{n-1}(x_1,\dots,x_{n-1})\diamond x_n,P_{n-1}(y_1,\dots,y_{n-1})\diamond y_n) \\
   & \le d(P_{n-1}(x_1,\dots,x_{n-1})\diamond x_n,P_{n-1}(y_1,\dots,y_{n-1})\diamond x_n) \\ &\qquad\qquad\qquad\qquad+d(P_{n-1}(y_1,\dots,y_{n-1})\diamond x_n,P_{n-1}(y_1,\dots,y_{n-1})\diamond y_n)\\
   &\le d(P_{n-1}(x_1,\dots,x_{n-1}),P_{n-1}(y_1,\dots,y_{n-1}))+d(x_n,y_n)\\
   &\le \sum_{i=1}^{n}d(x_i,y_i).
\end{align*}

To prove {\rm(ii)} one first writes
\[d(f_{a,b}(x),f_{a,b}(y))\le d(f_{a,b}(x),f_{c,d}(x))+d(f_{c,d}(x),f_{c,d}(y))+d(f_{c,d}(y),f_{a,b}(y)),\]
and then apply {\rm(i)} to obtain
\[d(f_{a,b}(x),f_{a,b}(y))\le d(f_{c,d}(x),f_{c,d}(y)) +2 d(a,c) +2d (b,d).\]
\end{proof}

\begin{prop}\label{barJ}
Let $\emptyset\neq J\subseteq J(L)$ and let $\tilde{J}:=\bar J\cap J(L)$, where $\bar J$ is  the closure of $J$ in  $(L\times L,u\times u)$. Then $u_J=u_{\tilde{J}}$.
\end{prop}
\begin{proof}
 The inclusion $u_{J}\subseteq u_{\tilde J}$ follows by Proposition \ref{uj}{\rm(ii)}.  For the reverse inclusion, suppose that $W\in u_{\tilde J}$.  We show that there exists $U\in u_{ J}$ such that $U\subseteq W$.   In view of Remark \ref{r1} there exist $d\in D_u$, $(a,b)\in \tilde J$ and  $\varepsilon>0$, such that $(x,y)\in W$ for every $(x,y)\in L^2$ satisfying $d_{a,b}(x,y)<5\varepsilon$.  Let $(a',b')\in J$ with $d(a,a')<\varepsilon$ and $d(b,b')<\varepsilon$.   Then, $U_{a',b'}:=\{(x,y)\in L^2:d_{a',b'}(x,y)<\varepsilon\}\in u_{J}$ and by  Lemma \ref{l2}{\rm(ii)}
\[d_{a,b}(x,y)\le d_{a',b'}(x,y)+2d(a,a')+2d(b,b')<5\varepsilon,\]
for every  $(x,y)\in U_{a',b'}$.  This implies that $U_{a',b'}\subseteq W$.
\end{proof}

\begin{cor}\label{denseS}
Let $S$ be a dense sublattice of $(L,u)$ and $v=u|_S$. Then $\restr{u^*}{S}=v^*$.
\end{cor}

\begin{proof}
Let $J:=J(S)=\{(a,b)\in S^2: a\leq b\}$. First observe that $J(L)\subseteq \tilde{J}$: If $(a,b)\in J(L)$ and $(x_\alpha)$, $(y_\alpha)$ are nets in $S$ converging in $(L,u)$ to $a$ and $b$, respectively, then
$J\ni (x_\alpha\wedge y_\alpha,x_\alpha\vee y_\alpha)\rightarrow (a,b)$ ($u$), hence $(a,b)\in \tilde{J}$.

Applying Proposition \ref{barJ} we obtain $u^*=u_J$. Obviously $\restr{u_J}{S}=v_J$.

Combining, we get $\restr{u^*}{S}=\restr{u_J}{S}=v_J=v^\ast$.
\end{proof}

We now use besides Corollary \ref{denseS} the following general observation:

\begin{lem}[\cite{Weber1982}, pg. 381]\label{general}
Let $Y$ be a dense subspace of a uniform space $(X,v)$.
\begin{enumerate}[{\rm(i)}]
\item Then $w\mapsto w|_Y$ defines a lattice isomorphism from the lattice of all uniformities on $X$ coarser than $v$ onto the lattice of all uniformities on $Y$ coarser than $v|_Y$.
\item  Let $w$ be a uniformity on $X$ coarser than $v$ and $y_0\in Y$. Then
$$\{\bar{W}^v: W \text{ is a neighbourhood of }y_0 \text{ in }(Y,w|_Y) \}$$
is a neighbourhood base of $y_0$ in $(X,w)$; here $\bar{W}^v$ denotes the closure of $W$ in $(X,v)$.
\end{enumerate}
\end{lem}

\begin{theo}\label{dense*}
Let $S$ be a dense sublattice of $(L,u)$ and $v=u|_S$.
\begin{enumerate}[{\rm(i)}]
\item  Then $u^*=u$ iff $v^*=v$.
\item  If the $u^*$-topology coincides with the $u$-topology, then $v^*$-topology coincides with the $v$-topology.
\end{enumerate}
\end{theo}

\begin{proof}
{\rm(i)}~By Lemma \ref{general}, $u^*=u$ iff $u^*|_S=u|_S$. Now observe that $u|_S=v$ by definition and $u^*|_S=v^*$ by Corollary \ref{denseS}.

{\rm(ii)}~If the $u^*$-topology coincides with the $u$-topology, then the $u^*|_S$-topology coincides with the $u|_S$-topology, i.e in view of Corollary \ref{denseS}, the $v^*$-topology coincides with the $v$-topology.
\end{proof}

We don't know whether in {\rm(ii) of Theorem \ref{dense*} the converse implication is true. Lemma \ref{general} {\rm(ii)} and Corollary \ref{denseS} only imply that every $s\in S$ has the same neighbourhood system in $(L,u^*)$ and in $(L,u)$ if $v^*$-topology coincides with the $v$-topology. But this is enough to answer the question of Taylor (see [T19, Question 3.3]) whether the property `unbounded' passes from the Hausdorff locally solid Riesz space $(X,\tau)$ to its completion $(\tilde{X},\tilde \tau)$.




\begin{theo}\label{densegroup}
Let $H$ be a dense solid subgroup of a locally solid commutative $\ell$-group $(G,\tau)$ and let $\sigma=\tau|_H$. Then $\mathfrak{u}\tau=\tau$ iff $\mathfrak{u}\sigma=\sigma$.
\end{theo}

\begin{proof}
Let $u$ and $v$ be the uniformities induced, respectively, by $\tau$ and $\sigma$. We use that the $u^*$-topology and the $v^*$-topology coincide, respectively, with $\mathfrak{u}\tau$ and $\mathfrak{u}\sigma$ (see Theorem \ref{cf}) and therefore $\mathfrak{u}\tau|_H=\mathfrak{u}\sigma$ (see Corollary \ref{denseS}).

If $\mathfrak{u}\tau=\tau$, then $\mathfrak{u}\sigma=\sigma$ by Theorem \ref{dense*} (b). Vice versa, if $\mathfrak{u}\sigma=\sigma$, then $\mathfrak{u}\tau$ and $\tau$ have the same $0$-neighbourhood system by Lemma \ref{general}(ii). Therefore  $\mathfrak{u}\tau=\tau$ since both, $\mathfrak{u}\tau$ and $\tau$, are group topologies.
\end{proof}

The following Theorem generalizes \cite[Proposition 2.12]{T} and answers \cite[Question 2.13]{T}.

\begin{theo}\label{t4}
Let $S$ be a sublattice of the uniform lattice $(L,u)$.  Then $S$ is (sequentially) closed w.r.t. $u^\ast$ iff it is (sequentially) closed w.r.t. $u$.
\end{theo}
\begin{proof}
For an infinite cardinal $\kappa$ let us say that a subset $U$ of a topological space is $\kappa$-closed when $U$ contains the limit of every convergent net $(x_\alpha)_{\alpha\in A}$ where $|A|\le \kappa$,  that is contained in $U$.  We show that $S$ is $\kappa$-closed w.r.t. $u^\ast$ iff it is $\kappa$-closed w.r.t. $u$.  Since $u^\ast\subseteq u$, clearly every $\kappa$-closed subset w.r.t. $u^\ast$ is $\kappa$-closed w.r.t. $u$.  For the converse, assume that $S$ is $\kappa$-closed w.r.t. $u$ and  suppose that $(s_\alpha)_{\alpha\in A}$ is  a net in $S$ such that $|A|\le \kappa$ and convergent to $x$ w.r.t. $u^\ast$.  For every $s\in S$ and $z\in L$ satisfying $z\ge s$ we observe that
$(s_\alpha\vee s)\wedge z\to(x\vee s)\wedge z$ w.r.t. $u$.  In particular, if we fix $\alpha'\in A$ and set $z:=s_{\alpha'}\vee s$, we obtain -- by the hypothesis that $S$ is $\kappa$-closed w.r.t. $u$ -- that $(x\vee s)\wedge (s_{\alpha'}\vee s)\in S$.  Since $\alpha'$ was arbitrary, and since $s_\alpha\to x$ w.r.t. $u^\ast$, this in turn implies that
\[(x\vee s)\wedge (s_\alpha\vee s)\to (x\vee s)\wedge (x\vee s)=x\vee s\qquad\text{w.r.t. $u$}.\]
Using again the fact that $S$ is $\kappa$-closed w.r.t. $u$  we deduce that

{\rm(i)} $x\vee s\in S$ for every $s\in S$.

{\rm(ii)} Dually one obtains that $x\wedge s'\in S$ for every $s'\in S$.

Let now $s\in S$. Then $s':=x\vee s\in S$ by {\rm(i)} and $x=x\wedge s'\in S$ by {\rm(ii)}.
\end{proof}

\begin{cor}\label{c2}
Let $S$ be a sublattice of  the uniform lattice $(L,u)$.  Then the closures of $S$ in $L$ w.r.t. $u$ and $u^\ast$ coincide.
\end{cor}

By an obvious modification of the proof of Theorem \ref{t4} one obtains an analogous result for unbounded order convergence, which generalizes \cite[Proposition 3.15]{GaTrXa2017}.

\begin{prop}\label{uo-o}
Suppose that $x_\alpha\uparrow x$ implies $x_\alpha\wedge y\uparrow x\wedge y$ and $x_\alpha\downarrow x$ implies $x_\alpha\vee y\downarrow x\vee y$ for every net $(x_\alpha)$ in $L$ and $x,y\in L$ (i.e. $L$ is continuous in the sense of von Neumann -- cf. \cite[Definition 2.14]{MaedaMaeda}). Then a sublattice $S$ of $L$ is $\mathfrak{u}O$-closed iff it is $O$-closed.
\end{prop}

\begin{proof}
$\Leftarrow$ can be proved as $\Leftarrow$ of Theorem \ref{t4} replacing $u^*$-convergence by $\mathfrak{u}O$-convergence and $u$-convergence by $O$-convergence. For $\Rightarrow$ observe that the additional assumption exactly means that $O$-convergence implies $\mathfrak{u}O$-convergence. Therefore any $\mathfrak{u}O$-closed subset of $L$ is $O$-closed.
\end{proof}

\begin{prop}\label{l4}
The following conditions are equivalent:
\begin{enumerate}[{\rm(i)}]
\item $u$ is locally exhaustive;
\item $u^\ast$ is locally exhaustive;
\item $u^\ast$ is exhaustive.
\end{enumerate}
\end{prop}

\begin{proof}
Since $u$ and $u^*$ coincide on bounded sets by Corollary \ref{weakest}, we have {\rm(i)}$\Leftrightarrow${\rm(ii)}. {\rm(iii)}$\Rightarrow${\rm(ii)} is obvious.

{\rm (i)}$\Rightarrow${\rm(iii)}: By definition, a net  $(x_\gamma)_{\gamma\in\Gamma}$ is $u^\ast$-Cauchy iff $(f_{a,b}(x_\gamma))_{\gamma\in\Gamma}$ is $u$-Cauchy for every $(a,b)\in J(L)$. Let now $(x_\gamma)_{\gamma\in\Gamma}$ be a monotone net in $L$ and $(a,b)\in J(L)$. Then $(f_{a,b}(x_\gamma))_{\gamma\in\Gamma}$ is monotone and bounded, hence $u$-Cauchy by (i), i.e. $(x_\gamma)_{\gamma\in\Gamma}$ is $u^\ast$-Cauchy.
\end{proof}

\begin{theo}
Let $u$ and $v$ be Hausdorff lattice uniformities on $L$ satisfying Condition (C). Then:
\begin{enumerate}[{\rm(i)}]
\item the $u^\ast$-topology and the $v^\ast$-topology are equal;
\item $\overline{S}^{u}=\overline{S}^v$ for every sublattice $S$ of $L$.
\end{enumerate}
\end{theo}

\begin{proof}
{\rm(i)} First observe that the lattice uniformities $u^\ast$ and $v^\ast$ also satisfy Condition (C) since $u^*\subseteq u$ and $v^*\subseteq v$.  Therefore $u^\ast$ and $v^\ast$ are locally exhaustive by Proposition \ref{7.1.2}, hence exhaustive in view of Proposition \ref{l4}.  Moreover, Corollary \ref{c1} yields that $u^\ast$ and $v^\ast$ are  Hausdorff.  So, {\rm(i)} follows by Theorem \ref{W93C7.2.4}.

{\rm(ii)} follows from {\rm(i)} and Corollary \ref{c2}.
\end{proof}

Note that {\rm(ii)} of the above theorem generalizes \cite[Theorem 5.11]{T}.


The following lemma is needed in the proof of Theorem \ref{t2} to reduce it to the Hausdorff case.

\begin{lem}\label{l3}  Let $(\hat{L},\hat{u})=(L/\sss{\sim_u}\,,\,\widehat u)$ be the Hausdorff uniform lattice associated with $(L,u)$ according to Proposition \ref{q}.  Then
\[\left(\hat{L}\,,\,\widehat{u^\ast}\right)=\left(\hat{L} \,, \,(\widehat{u})^\ast\right).\]

\end{lem}

\begin{proof}
Observe first that $N(u)=N(u^\ast)$ by Proposition \ref{p2}, i.e. $L/\sss{\sim_u}=L/\sss{\sim_{u^\ast}}$.  We show that $(\widehat{u})^\ast=\widehat{(u^\ast)}$.  We shall make use of the  following fact that follows immediately by  Proposition \ref{q}: If $v$ is a lattice uniformity on $L$ satisfying $N(u)=N(v)$, then
\[\widehat{V\cap[a,b]^2}=\widehat{V}\cap [\hat{a},\hat{b}]^2,\]
for every closed subset $V$ of $(L,v)^2$ and $a\le b$ in $L$.

Hence, for  every $a\le b$ in $L$ we have
\[\widehat{u^\ast}|_{[\hat{a},\,\hat{b}]}\,=\,\widehat{u^\ast|_{[a,\,b]}}\,=\,\widehat{u|_
{[a,\,b]}}=\widehat{u}|_{[\hat{a},\,\hat{b}]}\]
and therefore $\widehat{u^\ast}\supset \widehat{u}^\ast$.  Conversely, if $v$ is a lattice uniformity on $L$ with $N(v)=N(u)$ and $\hat{v}=\widehat{u}^\ast$, then $\widehat{v}|_{[\hat{a},\,\hat{b}]}=\widehat{u}|_{[\hat{a},\,\hat{b}]}$, and therefore $v|_{[a,\,b]}=u|_{[a,\,b]}$ for every $a\le b$.  This implies that $v\supset u^\ast$ and therefore $\widehat{u}^\ast=\widehat{v}\supset\widehat{u^\ast}$.
\end{proof}

Let $u$ be Hausdorff and $(\tilde L,\tilde u)$ the completion of $(L,u)$.  It is easy to see that $\tilde L$ is again distributive, when $L$ is distributive.   Moreover, both properties of local exhaustivity and exhaustivity pass directly to this completion by Proposition \ref{exh}.

\begin{theo}\label{t2}
Let $(L,u)$ be a uniform lattice.  If $u$ is exhaustive, then $u=u^\ast$.
\end{theo}

\begin{proof}
{\rm (i)} Let us first suppose that $u$ is Hausdorff.  Let $(\tilde L,\tilde u)$ denote the completion of $(L,u)$.  Then $\tilde L$ is a complete lattice by Corollary \ref{com} and therefore obviously $\tilde u=(\tilde u)^\ast$.  On the other hand, $\restr{(\tilde u)^\ast}{S}=u^\ast$ by Corollary \ref{denseS}.  Combining we get $u=\restr{\tilde u}{S}=\restr{(\tilde u)^\ast}{S}=u^\ast$.

{\rm (ii)} In general, if $u$ is exhaustive, $\widehat{u}$ is Hausdorff and exhaustive, and therefore $\widehat{u}=(\widehat{u})^\ast$ by (i).  Therefore, by Lemma \ref{l3}, we obtain $\widehat{u}=\widehat{u^\ast}$.  Since $N(u)=N(u^\ast)$, this implies\footnote{It is easy to see and follows from Proposition \ref{q} that if $u$ and $v$ are two lattice uniformities on $L$ satisfying $N(u)=N(v)$, then the induced uniformities $\widehat{u}$ and $\widehat{v}$ on the common quotient $L/\sss{\sim_u}=L/\sss{\sim_v}$ are equal iff $u$ and $v$ are equal.} that $u=u^\ast$.
\end{proof}

\begin{cor}\label{sub*}
Let $(L,u)$ be a uniform lattice and let $u$ be locally exhaustive.  Then $\restr{u^\ast}{S}=(\restr{u}{S})^\ast$ for every sublattice $S$ of $L$.
\end{cor}

\begin{proof}
By Proposition 	\ref{l4} $u^*$ is exhaustive. Therefore $u^*|_S$ is exhaustive, hence $(u^*|_S)^*=u^*|_S$ by Theorem \ref{t2}. Since $u|_S\supseteq u^*|_S$, we have $(u|_S)^*\supseteq (u^*|_S)^*$. Moreover, $u^*|_S\supseteq (u|_S)^*$ since $u^*|_S$ agrees with $u|_S$ on order-bounded subsets of $S$ (c.f. Proposition \ref{uj} (iii) ). We have seen:
$$(u|_S)^*\supseteq (u^*|_S)^*=u^*|_S\supseteq (u|_S)^*,$$
$\restr{u^\ast}{S}=(\restr{u}{S})^\ast$.
\end{proof}

 Since for locally solid Riesz spaces the Lebesgue property implies the pre-Lebesgue property,  Corollary \ref{sub*} generalizes \cite[Lemma 9.7]{T} (therefore \cite[Corollary 4.6]{KMT}).

\bibliographystyle{amsplain}
\providecommand{\bysame}{\leavevmode\hbox to3em{\hrulefill}\thinspace}
\providecommand{\MR}{\relax\ifhmode\unskip\space\fi MR }
\providecommand{\MRhref}[2]{%
  \href{http://www.ams.org/mathscinet-getitem?mr=#1}{#2}
}
\providecommand{\href}[2]{#2}

\end{document}